\documentclass[12pt]{article}
\usepackage{amssymb}
\usepackage{amsmath}
\usepackage{amsthm}
\usepackage{mathrsfs}
\usepackage{float}
\usepackage{url}
\usepackage{bm}
\usepackage{graphicx}
\usepackage{algorithm,algorithmic,subfigure}
\usepackage{subeqnarray}
\usepackage{cases}
\usepackage{authblk}
\usepackage{blindtext}
\usepackage{xcolor}
\usepackage{cleveref}

\newcommand{\st}{{\textrm{s.t.}}}

\newcommand{\Acal}{{\mathcal{A}}}

\newcommand{\Ical}{{\mathcal{I}}}
\newcommand{\Ncal}{{\mathcal{N}}}
\newcommand{\Scal}{{\mathcal{S}}}

\newcommand{\bx}{\bm{x}}

\newcommand{\be}{\bm{e}}

\newtheorem{theorem}{Theorem}[section]

\newtheorem{lemma}[theorem]{Lemma}

 \theoremstyle{definition}
\newtheorem{definition}[theorem]{Definition}
\theoremstyle{remark}

\providecommand{\keywords}[1]
{
	\small
	\textbf{\textit{Keywords---}} #1
}

\begin{document}
\date{}

\title{An Iteratively Reweighted Method  for Sparse Optimization on Nonconvex $\ell_{p}$ Ball}

\author{Hao~Wang}
\author{Xiangyu~Yang}
\author{Wei~Jiang}
\affil{School of Information Science and Technology, ShanghaiTech University, Shanghai 201210, China}
\affil{\textit {\{wanghao1,yangxy3,jiangwei1\}@shanghaitech.edu.cn}}
\maketitle

\begin{abstract}
This paper is intended to solve the nonconvex $\ell_{p}$-ball constrained nonlinear optimization problems. An iteratively reweighted method is proposed, which solves a sequence of weighted $\ell_{1}$-ball projection subproblems.  At each iteration,  the next iterate is obtained by moving along the negative gradient with a stepsize and then projecting the resulted point onto the weighted $\ell_{1}$ ball to approximate the $\ell_{p}$ ball. Specifically, if the current iterate is in the interior of the feasible set, then the weighted $\ell_{1}$ ball is formed by linearizing the $\ell_{p}$ norm at the current iterate.  If the current iterate is on the boundary of the feasible set,  then the weighted $\ell_{1}$ ball is formed differently by keeping those zero components in the current iterate still zero. In our analysis, we prove that the generated iterates converge to a first-order stationary point. Numerical experiments demonstrate the effectiveness of the proposed method.
\end{abstract}


\keywords{nonconvex regularization, sparse optimization, iterative reweighting method, weighted $\ell_1$-ball projection}


\section{Introduction}\label{Intro_sec}

The use of a bounded norm constraint has received considerable attention in constrained nonlinear programming because it covers a rich set of applications, including machine learning, statistics, compressed sensing, and many other related fields~\cite{das2013non,duchi2008efficient,oymak2017sharp,tibshirani1996regression}. Enforcing a norm ball generally results in sparse solutions of interest,   which can capture the low-dimensional structures of the high-dimensional problems~\cite{kyrillidis2015structured,vidal2016generalized}. These benefits are evidenced in many learning tasks in which sparsity is usually the prior information for the desired solutions~\cite{hastie2015statistical}. In the past decade, people have witnessed significant advances in various norm constrained optimization models and numerical algorithms for solving norm ball constrained optimization. For example, in deep learning applications, weight pruning is a widely used technique in model compression~\cite{deng2020model,wang2020proximal}, whose primary focus is to choose a set of representative weights to balance the compactness and accuracy of the huge-size deep neural networks. Toward this end, researchers often formulated this problem as the convex loss function constrained by a bounded norm (e.g., $\ell_1$ norm and $\ell_0$ norm)~\cite{oymak2018learning}. Likewise, in the context of sparse signal recovery with linear measurements, the unknown signals need to be estimated based on the given sensing matrix and noisy observations. This is usually achieved by solving a least-squares problem with norm balls~\cite{bahmani2013unifying}.

Although the sparse optimization problems constrained by the $\ell_0$ ball or $\ell_1$ ball often arise in numerous contexts~\cite{duchi2008efficient,blumensath2009iterative,condat2016fast,michelot1986finite,rodriguez2018accelerated,zhao2020optimal}, based on the empirical evidence presented in~\cite{oymak2017sharp}, the $\ell_p$ ball constrained problems with $p\in(0,1)$ usually shows its superior performance in many areas including sparse recovery   and phase transition   compared with the $\ell_1$ ball and $\ell_0$ ball constraints. Moreover, the benefits of using nonconvex $\ell_p$ ball in Euclidean projection are revealed in~\cite{das2013non} for image denoising and sparse principal component analysis and    \cite{chen2019outlier} for computer vision. However, until now, to our knowledge, the algorithms proposed for handling the problems constrained by $\ell_p$ ball with $p \in (0,1)$ are limited. This is mainly due to the nonsmooth and nonconvex nature of the feasible set, which makes the analysis challenging and hinders the development of numerical algorithms.

Many research efforts are devoted to designing algorithms for solving nonconvex set constrained minimization problems. For example, projected gradient descent (PGD) is considered for solving general nonconvex constrained optimization problem~\cite{jain2017non}, which requires the projection onto the feasible set for each iteration. As presented in~\cite{jain2017non}, PGD usually needs the global optimal solution of the projection onto the nonconvex set at each iteration. However, this is generally a nontrivial task; as shown in~\cite{yang2021towards},  finding a first-order stationary solution of the $\ell_p$ ball projection problem is NP-hard, let alone compute the global optimal solution. The authors in~\cite{blumensath2009iterative}   analyzed the convergence rates of the PGD  algorithm for solving the $\ell_p$ ball constrained least-squares, for which they also assumed that a practical algorithm could be used to compute the global optimal solution of the $\ell_p$ ball projections for each iteration, making it impractical for real-world applications.

In this paper, we propose an iteratively reweighted $\ell_1$ ball method for solving the minimization problem constrained by the nonconvex $\ell_{p}$ norm ball. The main idea of our approach is to use the weighted $\ell_1$ balls to approximate the $\ell_p$ ball so that our approach iteratively solves a sequence of tractable projection subproblems. At an iterate in the interior of the feasible set, the weighted $\ell_1$ ball is constructed by linearizing the $\ell_p$ norm with respect to the nonzero components in the current iterate and setting the weights associated with the zero components according to the constraint residuals. At an iterate on the boundary of the feasible set, we form the weighted $\ell_1$ ball by linearizing the constraint set for the nonzero components while fixing the rest components as zero.

The idea of using iteratively reweighted norms to approximate $\ell_p$ norm is not novel. However, the existing research efforts were largely devoted to solving the $\ell_p$-norm regularization problems~\cite{Chen2013OptimalityCA,Lu2014IterativeRM,wang2018nonconvex}, meaning the $\ell_p$ norm presents in the objective instead of the constraint.  To our knowledge, there has not been much work on designing iteratively reweighted methods for handling the case where the  $\ell_p$-norm appears in the constraint.   As a particular focus,  the iteratively reweighted $\ell_1$ ball projection algorithm was proposed in~\cite{yang2021towards} for solving the Euclidean projection of a vector onto the nonconvex $\ell_p$ ball.   Their main idea is to use a relaxing vector to smooth the $\ell_p$ norm at each iteration, and then the weighted $\ell_1$ balls are obtained by linearizing the relaxed constraints. Their method requires carefully driving the relaxation parameter to zero to enforce convergence. Our method does not involve a relaxation parameter.

We summarize the contributions of this paper in the following.
\begin{itemize}
	\item[1.]  The problem with $\ell_p$ ball constraint considered in our paper has general form. Investigating such a problem offers the possibility to solve a broader class of objective functions  that can include various applications.
	
	\item[2.] We propose an iteratively reweighted $\ell_1$ ball method for handling the $\ell_{p}$-ball constrained minimization, which is easy to implement and computationally efficient. At each iteration, our algorithm only needs to solve weighted $\ell_1$-ball projection subproblems. The effectiveness of the proposed method is exhibited in the numerical studies via sparse signal recovery and logistic regression problems.
	
	\item[3.] We establish the global convergence analysis and show that the sequence generated by our method converges to the first-order optimal solutions from any feasible initial points.
\end{itemize}

\section{Preliminaries}

Throughout this paper, we use $\mathbb{R}^n$ to denote the Euclidean $n$-space   with the dot product defined by $\langle\bm{x}, \bm{y}\rangle = \bm{x}^{T}\bm{y} =\sum_{i=1}^{n}x_iy_i$.  Moreover, $\circ$ defines a component-wise product between $\bm{x}, \bm{y} \in\mathbb{R}^{n}$, i.e. $\bm{z}=\bm{x}\circ\bm{y}$ with $z_i=x_iy_i$ for $i\in[n]$.    The nonnegative orthant  of $\mathbb{ R}^n$  is denoted as $\mathbb{R}^n_+$ and the interior of $\mathbb{R}^n_{+}$ is denoted as $\mathbb{R}^n_{++}$. For $\bm{x}\in\mathbb{R}^n$, the $i$th component of $\bm{x}$ is represented by $x_i$ and $|\bm{x}|=[|x_1|,\dots,|x_n|]^T$. In particular, we use $\bm{e}$ and $\bm{0}$ to denote the vectors of all ones and zeros of appropriate size, respectively. Moreover, for $p > 0$, the $\ell_p$ (quasi)norm of $\bm{x}\in\mathbb{R}^{n}$ is defined as $\|\bm{x}\|_p=(\sum_{i=1}^n|x_i|^p)^{\frac{1}{p}}$.  Denote $\mathbb{N}$ as the set of natural numbers. We use $[n] \subset \mathbb{N}$ to represent a set containing the natural numbers from $0$ to $n$, i.e., $[n] = \{0,1,2,\dots,n\}$.  The cardinality of an index set $\mathcal{S}\subset \mathbb{ N }$ is denoted by $|\mathcal{ S }|$. We annotate $\bm{x}_{\mathcal{S}} \in \mathbb{R}^{|\mathcal{S}|}$ as the subvector of $\bm{x}$ where the components are indexed by $\mathcal{ S }$. The support set of $\bm{x}\in\mathbb{R}$ is defined as $\Ical(\bm{x}) := \{i \in [n] \mid x_{i} \neq 0\}$, and its complementary set is denoted as $\Acal(\bm{x}) := \{i \in [n] \mid 
x_{i} = 0\}$. The signum function of $x \in \mathbb{R}$ is defined as $\text{sgn}(x) = -1$ if $x<0$, $\text{sgn}(x) = 0$ if $x = 0$ and $\text{sgn}(x) = 1$ if $x > 0$. In addition, for $\bm{x} \in \mathbb{R}^{n}$, $\text{sgn}(\bm{x}) = [\text{sgn}(x_1),\ldots,\text{sgn}(x_n)]^{T}$.

\section{Algorithm description}\label{Sec: Algorithm_description} 

In this paper, we focus on solving the nonlinear optimization problem involving the nonconvex $\ell_p$ ball constraint, which is formalized as
\begin{equation}\label{lp.prob}\tag{\text{$\mathscr{P}$}}
	\begin{aligned}
		\min_{\bm{x}} \quad&f(\bm{x})\\ 
		\st \quad&\bm{x}\in \Omega := \{\bm{x}\in\mathbb{R}^{n}\mid\|\bm{x}\|_p^p\leq  r\},
	\end{aligned}
\end{equation}
where $f:\mathbb{ R }^n\rightarrow\mathbb{ R }$ is twice  continuously differentiable,   $p \in (0,1)$  and $r >0$ is referred to as the radius of the $\ell_p$ ball. 
We know $f$ is (locally) $L$-Lipschitz differentiable on $\Omega$.

Our algorithm alternatively solves two types of subproblems based on whether  the current iterate is on the boundary of $\Omega$. At the $k$th iteration, we approximate the objective $f$ by 
\begin{equation*}
	\phi(\bm{x};\bm{x}^k):= f(\bm{x}^k)+\nabla f(\bm{x}^k)^T(\bm{x}-\bm{x}^k)+\frac{\beta }{2}\|\bm{x}-\bm{x}^k\|_2^2, 
\end{equation*}
with $\beta > L$.    Specifically, if $\|\bx^k\|_p^p = r$,  the $\ell_p$ ball is linearized at $\bx^k$, so that we require the next iterate to satisfy 
\begin{equation*}
	\|\bx^k\|_p^p +   \sum_{ i\in\Ical(\bm{x}^{k}) } p|x_i^k|^{p-1}(|x_i|-|x_i^k|) \le r \quad\text{ and }\quad 
	x_i = 0,\quad i\in\mathcal{ A }(\bm{x}^{k}),
\end{equation*}
or, equivalently, solving the following subproblem for $\bx^{k+1}$, 
\begin{equation}\label{sub.1}\tag{\text{$\mathcal{P}_1$}}
	\begin{aligned} 
		\min_{\bx} & \ \  \phi(\bm{x};\bm{x}^k)\\
		\text{s.t.} &     \sum_{ i\in\Ical(\bm{x}^{k}) }  |x_i^k|^{p-1} |x_i|  \le  r, \\ 
		& \quad x_i = 0,    i\in\mathcal{ A }(\bm{x}^{k}).
\end{aligned}\end{equation}
If the current iterate is in the interior of $\Omega$, i.e., $\|\bx^k\|_p^p < r$, we use the constraint residual as the perturbation parameter 
to the zero components and then formulate a weighted $\ell_1$ ball by linearizing the $\ell_p$ ball at $\bx^k$. To achieve this,  
letting
\begin{equation}\label{set.epsilon}
	\epsilon^k = c\Big(\tfrac{ r -  \|\bx^k\|_p^p  }{|\Acal(\bx^k)| + 1}\Big)^{\frac{1}{p}},
\end{equation}
\begin{equation}\label{set.rk}
	\begin{aligned}
		r^{k} :=  \tfrac{1}{p}( r + (p-1) \|\bx^k   \|_p^p - | \Acal(\bx^k)|  \epsilon^p )
		> \frac{1-c^p}{p}( r -  \|\bx^k\|_p^p ) >0
	\end{aligned}
\end{equation}
with $c\in(0,1)$. We require the next iterate to satisfy 
\begin{equation*}\label{eq: interior_constraint}
	\begin{aligned}
		\|\bx^k_{\Ical(\bx^k)}   \|_p^p + | \Acal(\bx^k)|  \epsilon^p +  \sum_{ i\in\Ical(\bm{x}^{k}) } p|x_i^k|^{p-1}(|x_i|-|x_i^k|) +\sum_{ i\in\mathcal{ A }(\bm{x}^{k}) } p (\epsilon^k)^{p-1}|x_i|\le r. 
	\end{aligned}
\end{equation*}
In other words, we solve the following subproblem for $\bx^{k+1}$, 
\begin{equation}\label{sub.2}\tag{\text{$\mathcal{P}_2$}}
	\begin{aligned} 
		\min_{\bx} & \quad   \phi(\bm{x};\bm{x}^k)\\
		\text{s.t.} &        \sum_{ i\in\Ical(\bm{x}^{k}) } |x_i^k|^{p-1} |x_i|  \!+\!\sum_{ i\in\mathcal{ A }(\bm{x}^{k}) }  (\epsilon^k)^{p-1}|x_i|\le r^k.
\end{aligned}\end{equation}
We state our proposed  Iterative Reweighted $\ell_1$ Ball Algorithm,  hereinafter named IR1B, in~Algorithm~\ref{alg.l1}.  

\begin{algorithm}[htbp]
	\caption{Iterative Reweighted $\ell_1$ Ball Algorithm (IR1B)}\label{alg.l1}
	\begin{algorithmic}[1]
		\STATE \textbf{Input:}~Choose $\bm{x}^0\in\Omega$, $\beta >\frac{L}{2}$, $c\in(0,1)$ and tolerance $\text{tol} \ge 0$.
		\REPEAT 
		\IF{$\|\bm{x}^k\|_p^p =  r$}
		\STATE Solve \eqref{sub.1} for $\bx^{k+1}$.  
		\ENDIF 
		\IF{$\|\bm{x}^k\|_p^p<r$}
		\STATE  Set $\epsilon^k$ according to \eqref{set.epsilon}. 
		\STATE Solve \eqref{sub.2} for $\bx^{k+1}$. 
		\ENDIF
		\STATE Set $k\gets k+1$. 
		\UNTIL{ $\|\bx^{k+1}-\bx^k\|_2 \le \text{tol}$ }
	\end{algorithmic}
\end{algorithm}

For each iteration of IR1B,  a weighted $\ell_1$ ball projection subproblem needs to be solved.   We mention that there have been many 
efficient $\ell_1$-ball projection algorithms proposed with observed complexity $O(n)$ in the past decades, e.g.,   \cite{condat2016fast,liu2009efficient,rodriguez2017accelerated,zhang2020inexact}.  These algorithms can be modified easily and extended to solve the weighted $\ell_1$-ball projection subproblems here.  We choose the algorithm proposed by~\cite{zhang2020inexact}, modify and use it as our subproblem solver. The details of this algorithm are therefore skipped.  We are aware that the performance of the proposed algorithm can be further improved by carefully designing a more efficient subproblem solver, which, however, is not the primary focus of this paper.

\section{Convergence analysis}\label{Sec: Convergence}

The convergence properties of IR1B are the subject of this section.  We first provide the first-order necessary optimality conditions to characterize 
the optimal solutions of \eqref{lp.prob}, and then prove the well-posedness of the algorithm along with global and local convergence 
results.  

\subsection{First-order optimality condition}\label{sec:3}
To characterize the optimal solution of \eqref{lp.prob}, we introduce the concept of Fr\'echet normal cone \cite{mehlitz2018limiting}, which is a generalization of the normal cone for closed convex sets.
\begin{definition}\label{def:FrechetNormalCone}
	(Fr\'echet normal cone) Given a closed set $\mathcal{X}\subset \mathbb{ R}^n$ and $\bm{x}^*\in \mathcal{X}$, the Fr\'echet normal cone of $\mathcal{X}$ at $\bm{x}^*$ can be defined as
	$$\Ncal_{\mathcal{X}}(\bm{x}^{*}) := \{\bm{\eta}\in \mathbb{ R}^n \mid  \limsup_{\bm{x} \rightarrow \bm{x}^{*}\bm{x} \in \mathcal{X}} \frac{ \langle\bm{\eta}, \bm{x}-\bm{x}^{*}\rangle}{\|\bm{x}-\bm{x}^{*}\|} \leq 0 \}.$$
\end{definition}
\noindent

We now investigate the Fr\'echet normal cone of  the  $\ell_p$ ball at 
$ \bar\bx\in \mathbb{R}^n$. The following theorem characterizes the elements of $\Ncal_{\Omega}(\bar\bx)$, and its proof can be found in~Appendix~\ref{Appendix_1}.

\begin{lemma}\label{thm:FirstOrderNecessaryCondition}  
	If  $\bar{\bm{x}} \in \mathbb{R}^n $ satisfies   $\|\bar{\bm{x}}\|_p^p< r$, 
	then $\Ncal_\Omega(\bar\bx) = \{\bm{0}\}$. 
	If    $\bar{\bm{x}} \in \mathbb{R}^n $ satisfies   $\|\bar{\bm{x}}\|_p^p= r$, 
	then $\bm{\eta}\in\Ncal_{\Omega}(\bar{\bm{x}})$  with  
	\begin{equation*}
		\eta_{i}= \begin{cases} 
			{p |\bar{x}_{i}|^{p-1}\text{sgn}(\bar{x}_i),} & {\text {if } i \in \mathcal{I}(\bar{\boldsymbol{x}})}, \\
			{\eta_{i} \in \mathbb{R},} & {\text {if } i \in \mathcal{A}(\bar{\bm{x}})}. 
		\end{cases}
	\end{equation*}
\end{lemma}

\noindent Using   Fr\'echet normal cone   and \cite[Theorem 4.3]{mordukhovich2006frechet}, we have the following 
necessary optimality condition  to \eqref{lp.prob}.

\begin{theorem}[Fermat's rule]\label{thm:optimalSolution}
	If \eqref{lp.prob} has a local minimum at $\bar\bx$, then it holds that
	\begin{equation}\label{eq: Fermat's rule}
		0\in \nabla f(\bar\bx)+\Ncal_{\Omega}(\bar\bx).
	\end{equation}
\end{theorem}

\noindent We call a point satisfying the necessary optimality condition~\eqref{eq: Fermat's rule} as a first-order optimal solution of~\eqref{lp.prob}.

The optimal solution $\bx^*$ of \eqref{lp.prob} can be classified  into two cases:   $\bx^*$ is on the boundary of the $\ell_p$ ball  or 
$\bx^*$    is in the interior of  the $\ell_p$ ball. 
For the first case, since  the $\ell_p$ norm constraint is active,  we have  from  \Cref{thm:FirstOrderNecessaryCondition} that any 
$\bx$ satisfying the following conditions is first-order optimal 
\begin{subequations}\label{kkt.1}
	\begin{alignat}{2}
		\nabla_i f(\bx)+\lambda p |x_i|^{p-1}sgn(x_i)&=0,\quad i\in\mathcal{ I }(\bx)\label{eq:subkkttmpab}\\
		\|\bx\|_p^p&= r,\label{eq:subkkttmpac}\\
		\lambda&\ge 0.\label{eq:originkktc}
	\end{alignat}
\end{subequations}
Here \eqref{eq:subkkttmpab} can be equivalently written as
\begin{equation}\label{eq: equivalent_condition}
	\nabla_i f(\bx) x_{i}+\lambda p |x_i|^{p} =0,\quad i\in\mathcal{ I }(\bx).
\end{equation}
For the second case,  the first-order necessary optimal condition is simply  
\begin{subequations}\label{kkt.2} 
	\begin{alignat}{1}
		\nabla f(\bx)&=0,\label{eq:originkkta_inside}\\
		\|\bx\|_p^p&< r.\label{eq:originkktb_inside}
	\end{alignat}
\end{subequations}
A primal-dual pair $(\bx,\lambda)$ satisfying \eqref{eq:originkkta_inside}-\eqref{eq:originkktb_inside} or \eqref{eq:subkkttmpab}-\eqref{eq:originkktc} is called the first-order optimal solution.

\subsection{Well-posedness}  

We prove that IR1B is well-posed in the sense that each iteration is well-defined.  Our first lemma reveals that  if  $\bx^k$ is feasible, then  the $k$th subproblems \eqref{sub.1} and \eqref{sub.2} 
are all feasible   and the next iterate $\bx^{k+1}$ is also feasible.  
\begin{lemma}\label{lem.well.1} 
	Given any $\bm{x}^k\in\Omega$. Subproblems \eqref{sub.1} and \eqref{sub.2} are all feasible. Moreover, $\|\bx^{k+1}\|_p^p \le r$, meaning 
	$\{\bx^k\} \subset \Omega$. 
\end{lemma}
\begin{proof} Obviously,  \eqref{sub.1} and \eqref{sub.2} are all feasible.  
	Suppose   \eqref{sub.1}  is solved at the $k$th iteration. By concavity of $(\cdot)^p$, we have 
	\begin{equation*}
		\begin{aligned}
			\|\bx^{k+1}\|_p^p = \| \bx^{k+1}_{\Ical(\bx^k)}\|_p^p \le \|\bx^k_{\Ical(\bx^k)}\|_p^p + \sum_{i\in \Ical(\bx^k)}p|x_i^k|^{p-1}(|x_i^{k+1}| - |x_i^k|) \le r.
		\end{aligned}
	\end{equation*}
	Suppose \eqref{sub.2} is solved at the $k$th iteration.
	By concavity, 
	\begin{equation*}
		\begin{aligned} 
			&\|\bx^{k+1}   \|_p^p
			\le  \|\bx^{k+1}_{\Ical(\bx^k)}   \|_p^p +  \| |\bx^{k+1}_{\Acal(\bx^k)} | +  \epsilon\be\|_p^p \\ 
			\le& \|\bx^k_{\Ical(\bx^k)}   \|_p^p +  \|| \bx^k_{\Acal(\bx^k)}| +  \epsilon\be\|_p^p + \!\sum_{i\in \Ical(\bx^k)}p|x_i^k|^{p-1}(|x_i^{k+1}| \!-\! |x_i^k|)\\
			&+ \! \sum_{i\in \Acal(\bx^k)}p\epsilon^{p-1}(|x_i^{k+1}| \!-\! |x_i^k|)  \\ 
			= & \|\bx^k_{\Ical(\bx^k)}   \|_p^p +  |\Acal(\bx^k)| \epsilon^p + \sum_{i\in \Ical(\bx^k)}p|x_i^k|^{p-1}(|x_i^{k+1}| - |x_i^k|)\\ 
			&+  \sum_{i\in \Acal(\bx^k)}p\epsilon^{p-1}|x_i^{k+1}|   
			\le  \   r.   
		\end{aligned}
	\end{equation*}
\end{proof}

For ease of notation, denote $\Ical^{k} := \Ical(\bm{x}^{k})$ and $\Acal^{k} := \Acal(\bm{x}^{k})$. Notice that the Slater conditions holds for \eqref{sub.1} and \eqref{sub.2}.  
For \eqref{sub.1}, the new iterate $\bx^{k+1}$ satisfies the following     Karush-Kuhn-Tucker (KKT) conditions of \eqref{sub.1}  
\begin{subequations}\label{kkt.sub1}
	\begin{alignat}{1}
		&\nabla_i f(\bm{x}^k)\!+\!\beta  (x_i^{k+1} -x_i^k)\!+\!\lambda^{k+1} p|x_i^k|^{p-1} \xi_i^{k+1}  =0, \ i\in \Ical^k\label{eq:sub_inside_kkt_a}\\
		&\lambda^{k+1}( \sum_{ i\in \mathcal{I}(\bm{x}^k)} p|x_i^k|^{p-1}|x_i^{k+1}| -  r ) =0, \label{eq:subkktc} \\
		&\sum_{ i\in \mathcal{I}(\bm{x}^k)}   |x_i^k|^{p-1}  |x_i^{k+1}| \le  r,\\
		&\lambda^{k+1}\ge 0,\quad\xi_i^{k+1} \in \partial |x_i^{k+1}|,
	\end{alignat}
\end{subequations} 
where $\lambda^{k+1}$ is the multiplier associated with the weighted $\ell_1$ ball constraint. Clearly, $\lambda^{k+1}$  satisfies
\begin{equation}\label{eq:lambdak1}
	\lambda^{k+1} = -\frac{\sum\limits_{i\in\mathcal{I}(\bm{x}^k)} [x_i^{k+1}\nabla_i f(\bm{x}^k)+\beta  x_i^{k+1} (x_i^{k+1}-x_i^k)] } { p \sum\limits_{i\in\Ical(\bx^k)} | x_i^k|^{p-1}|x_i^{k+1}|}.
\end{equation}
As for \eqref{sub.2}, the new iterate $\bx^{k+1}$ satisfies 
\begin{subequations}\label{kkt.sub2}
	\begin{alignat}{2}
		&\nabla_i f(\bm{x}^k)+\beta  (x_i^{k+1} -x_i^k) + \lambda^{k+1} p|x_i^k|^{p-1} \xi_i =0, \ i\in \Ical^k \\
		&\nabla_i f(\bm{x}^k)+\beta  (x_i^{k+1} -x_i^k) + \lambda^{k+1} p(\epsilon^k)^{p-1} \xi_i =0, \ i\in \Acal^k \\
		&\lambda^{k+1}(  \sum_{ i\in\Ical(\bm{x}^{k}) } |x_i^k|^{p-1} |x_i^{k+1}|  +\sum_{ i\in\mathcal{ A }(\bm{x}^{k}) }  (\epsilon^k)^{p-1}|x_i^{k+1}| -  r^k ) =0, \\
		&\sum_{ i\in\Ical(\bm{x}^{k}) }  |x_i^k|^{p-1} |x_i^{k+1}|  +\sum_{ i\in\mathcal{ A }(\bm{x}^{k}) }   (\epsilon^k)^{p-1}|x_i^{k+1}|\le r^k,\\
		&\lambda^{k+1}\ge 0,\quad	\xi_i^{k+1} \in \partial |x_i^{k+1}|, 	  
	\end{alignat}
\end{subequations} 
where $\lambda^{k+1}$ is the multiplier associated with the weighted $\ell_1$ ball constraint. Clearly, $\lambda^{k+1}$  satisfies
\begin{equation}\label{eq:lambdak2}
	\lambda^{k+1} = -\frac{\sum_{i=1}^n [x_i^{k+1}\nabla_i f(\bm{x}^k)+\beta  x_i^{k+1} (x_i^{k+1}-x_i^k)] }{ p \sum\limits_{i\in\Ical(\bx^k)} | x_i^k|^{p-1}|x_i^{k+1}| + p\epsilon^{p-1} \sum\limits_{i\in\Acal(\bx^k)}|x_i^{k+1}|}.
\end{equation}

The next lemma enumerates relevant properties of subproblem \eqref{sub.1} and \eqref{sub.2}.  It states that 
as long as $\bx^k$ is not stationary for the $k$th subproblem, the next iterate $\bx^{k+1}$ will cause a reduction for $f$ from $\bx^k$. 

\begin{lemma}\label{lem.well.2}  During the $k$th iteration of IR1B, we have the following: 
	\begin{enumerate}
		\item[(a)] If $\bx^{k+1} = \bx^k$ after solving \eqref{sub.1}, then $\bx^k$ satisfies  conditions \eqref{eq:subkkttmpab}-\eqref{eq:originkktc}. 
		\item[(b)] If $\bx^{k+1} = \bx^k$ after solving \eqref{sub.2}, then $\bx^k$ satisfies conditions \eqref{eq:originkkta_inside}-\eqref{eq:originkktb_inside}.
	\end{enumerate}
\end{lemma} 

\begin{proof} (a) If $\bx^{k+1} = \bx^k$ after solving \eqref{sub.1}, plugging $\bx^{k+1} = \bx^k$ into \eqref{kkt.sub1},  we have $\bx^k$ satisfies 
	\eqref{kkt.1}. Therefore, $\bx^k$ is first-order optimal for \eqref{lp.prob}. 
	
	(b) If $\bx^{k+1} = \bx^k$ after solving \eqref{sub.2}, we plug $\bx^{k+1} = \bx^k$ into \eqref{kkt.sub2}. The constraint is inactive at $\bx^k$, since 
	\[ \begin{aligned}
		&\ \|\bx^k_{\Ical(\bx^k)}   \|_p^p + | \Acal(\bx^k)|  \epsilon^p +  \sum_{ i\in\Ical(\bm{x}^{k}) } p|x_i^k|^{p-1}(|x_i^k|-|x_i^k|) +\sum_{ i\in\mathcal{ A }(\bm{x}^{k}) } p (\epsilon^k)^{p-1}|x_i^k|\\
		= &\ \|\bx^k   \|_p^p + | \Acal(\bx^k)|  \epsilon^p 
		=  \  \|\bx^k   \|_p^p + | \Acal(\bx^k)|  c^p \big(\tfrac{ r -  \|\bx^k\|_p^p  }{|\Acal(\bx^k)|+1}\big)\\
		< & \  \|\bx^k   \|_p^p +    \big( r -  \|\bx^k\|_p^p  \big) = r. 
	\end{aligned} \]
	Therefore, $\lambda^{k+1} = 0$ and   \eqref{kkt.sub2} now reverts to $\nabla f(\bx^k) = \bm{0}$, meaning $\bx^k$ is first-order optimal 
	for \eqref{lp.prob}. 
\end{proof}

Our main theorem in this subsection summarizes the well-posedness of IR1B, which 
can be derived trivially using \Cref{lem.well.1}  and \Cref{lem.well.2}. 

\begin{theorem}\label{thm.well}  Assume the tolerance in IR1B is set as $\text{tol} = 0$.   During the iteration of IR1B, one of the following must occur  
	\begin{enumerate}
		\item[(a)] IR1B terminates after solving subproblem \eqref{sub.1} with $(\bx^{k+1}, \lambda^{k+1})$ satisfying \eqref{kkt.1}. 
		\item[(b)] IR1B terminates after solving subproblem \eqref{sub.2} with $\bx^{k+1}$ satisfying \eqref{kkt.2}. 
		\item[(c)] IR1B generates an infinite sequence $\{(\bx^k, \lambda^k)\}$, where, for all $k$, 
		$\{\bx^k\}\subset \Omega$ and $\{\lambda^k\}\subset \mathbb{R}_+$. 
	\end{enumerate} 
\end{theorem}

\subsection{Global convergence} 

We now prove properties related to the global convergence of IR1B under the assumption that an infinite sequence of iterates 
is generated; i.e., we focus on the situation described in \Cref{thm.well}(c).  We first show that 
the objective of \eqref{lp.prob} is monotonically decreasing. 
For this purpose, we define the decrease in $\phi(\cdot ; \bx^k)$ caused by $\bx^{k+1}$ from $\bx^k$ as 
\[ \Delta \phi(\bx^{k+1}; \bx^k) = \phi(\bx^k; \bx^k) - \phi(\bx^{k+1}; \bx^k).\]
We first show this reduction vanishes as $k\to\infty$. 

\begin{lemma}\label{lem.delta.0}
	Suppose $\{\bx^k\}$ is generated by IR1B. It holds true that 
	\[f(\bx^k) -  f(\bx^{k+1}) \ge    \Delta \phi(\bx^{k+1}; \bx^k).\]
	Therefore, $\{f(\bx^k)\}$ is monotonically decreasing and $\lim\limits_{k\to\infty}  \Delta \phi(\bx^{k+1}; \bx^k)  = 0$. 
\end{lemma}
\begin{proof}
	Since $f$ is locally Lipschitz differentiable  on $\Omega$ and $\beta  > L$, we have 
	\[ \begin{aligned}
		f(\bx^{k+1})\
		&\le\ f(\bx^k) \!+\! \langle \nabla f(\bx^k), \bx^{k+1} - \bx^k\rangle \!+\! \frac{L}{2}\|\bx^{k+1} - \bx^k\|_2^2 \\ 
		&<\ f(\bx^k) \!+\! \langle \nabla f(\bx^k), \bx^{k+1} - \bx^k\rangle \!+\! \frac{\beta}{2}\|\bx^{k+1} - \bx^k\|_2^2 \\ 
		&=\ f(\bx^k) - \Delta \phi(\bx^{k+1}; \bx^k),
	\end{aligned}
	\]
	then we have 
	\begin{equation}\label{eq: reduction_bound}
		\Delta \phi(\bx^{k+1}; \bx^k) \le  f(\bx^k) - f(\bx^{k+1}).
	\end{equation}
	Rearranging, this leads to the desired result. On the other hand, summing both sides of~\eqref{eq: reduction_bound} from $t=0$ to $k$ gives
	\begin{equation}\label{eq: summing}
		\begin{aligned}
			\sum_{t=0}^{k} \Delta \phi(\bx^{t+1}; \bx^t)&\leq   \sum_{t=0}^{k}(f(\bx^t) - f(\bx^{t+1})) = f(\bm{x}^{0}) - f(\bm{x}^{t+1})\\ &\leq f(\bm{x}^{0}) - \underline{f}< +\infty
		\end{aligned}
	\end{equation}
	with $\underline{f}:= \inf\limits_{\bm{x} \in\Omega} f(\bm{x})$. Letting $t \to \infty$, we have
	\begin{equation*}
		\lim_{k\to\infty} \Delta \phi(\bm{x}^{k+1};\bm{x}^{k}) = 0. 
	\end{equation*}
\end{proof}
The next lemma shows that the displacement of the iterates vanishes in the limit.
\begin{lemma}\label{lem. iterates_dimishes}
	Suppose $\{\bx^k\}$ is generated by~Algorithm~\ref{alg.l1}. Then
	\begin{equation}\label{eq: iterates_dimishes}
		\lim\limits_{k\to \infty} \Vert\bm{x}^{k+1} - \bm{x}^{k}\Vert_{2} = 0.
	\end{equation}
\end{lemma}

\begin{proof} 
	By~\cite[Theorem 3.7]{zhang2020inexact}, we have $\Vert\bm{x}^{k+1} - \bm{x}^{k}\Vert_{2}^2 \leq (2/\beta) \Delta \phi(\bm{x}^{k+1};\bm{x}^{k})$. Therefore, we have $\lim\limits_{k\to \infty} \Vert\bm{x}^{k+1} - \bm{x}^{k}\Vert_{2} = 0$ by~\Cref{lem.delta.0}, as desired.
\end{proof}

Define two subsequences based on whether   \eqref{sub.1}  or \eqref{sub.2} is solved, 
\[ \Scal_1 = \{ k \mid \|\bx^k\|_p^p = r\}\ \text{ and } \ \Scal_2 = \{ k \mid \|\bx^k\|_p^p < r\}.\] 
We are now ready to provide the global convergence result for IR1B.  

\begin{theorem}\label{thm.global}
	Let $\{ (\bm{x}^{k},\lambda^k) \}$ be generated by~Algorithm~\ref{alg.l1}. Then, every cluster point $(\bx^*, \lambda^*)$ of $\{ (\bm{x}^{k},\lambda^k) \}$ is first-order optimal for \eqref{lp.prob}.
\end{theorem}

\begin{proof}
	By \Cref{lem.well.2},  if $\bx^*$ is a limit point of $\Scal_1$, it suffices to show that the subproblem \eqref{sub.1} at $\bx^*$ has a stationary point $\bx^*$;   if $\bx^*$ is a limit point of $\Scal_2$, it suffices to show that the subproblem  \eqref{sub.2} at $\bx^*$ has a stationary point $\bx^*$.  
	We prove this by contradiction for two cases.  
	
	Case (i): assume by contradiction that there exists a limit point  $\bx^*$ of $\Scal_1$ such that at $\bx^*$ the subproblem \eqref{sub.1}  has 
	a stationary point $\hat\bx \ne \bx^*$ and that  
	\begin{equation}\label{eq: Delta_p*}
		\Delta  \phi(\hat\bx; \bx^*) = \phi(\bx^*; \bx^*) - \phi(\hat\bx; \bx^*) > \delta > 0.
	\end{equation}
	Consider a subsequence $\hat\Scal_1 \subset \Scal_1$ such that $\{\bx^k\}_{\hat\Scal_1} \to \bx^*$.  
	Notice by \Cref{lem.delta.0}, there exists $k_0 \in \mathbb{N}$ such that 
	\[ \Delta \phi(\bx^{k+1}; \bx^k) = \phi(\bx^k; \bx^k) - \phi(\bx^{k+1}; \bx^k) < \delta/4,\] 
	or, equivalently, 
	\begin{equation}\label{delta is small} 
		\phi(\bx^{k+1}; \bx^k) > \phi(\bx^k; \bx^k) - \delta/4
	\end{equation} 
	for all $k > k_0$, $k\in \hat\Scal_1$. 
	To derive a conclusion contradicting \eqref{delta is small}, first of all, notice that 
	$\phi(\bx^k; \bx^k) - \phi(\bx^*; \bx^*)  = f(\bx^k) - f(\bx^*)$. 
	By the continuity of $f$, there exists $k_1$ such that for all $k > k_1$, $k\in \hat\Scal_1$
	\begin{equation}\label{eq: Delta_p^k}
		|\phi(\bx^k; \bx^k) - \phi(\bx^*; \bx^*)|  < \delta/4.
	\end{equation}
	Denote $\hat{\bx}^k$ as the projection of $\hat{\bx}$ onto the feasible region of~\eqref{sub.1} at $\bx^k$. By the continuity of $\phi$, there exists $k_2$ such that for all $k > k_2$, $k \in \hat\Scal_1$,
	\begin{equation}\label{eq: Delta_hat_x}
		\begin{aligned}
			|\phi(\hat{\bx}^k;\bm{x}^*) - \phi(\hat{\bx}^k;\bm{x}^{k})| < \delta / 4\quad\text{and}\quad |\phi(\hat{\bx}^k;\bm{x}^*) - \phi(\hat{\bx};\bm{x}^{*})| < \delta/4.
		\end{aligned}
	\end{equation}
	Combining~\eqref{eq: Delta_p*},~\eqref{delta is small},~\eqref{eq: Delta_p^k} and~\eqref{eq: Delta_hat_x}, for any $k > \max\{k_{0}, k_{1}, k_{2}\}$, we have
	\begin{equation}\label{eq: Contra_Delta_p}
		\begin{aligned}
			\phi(\bm{x}^{k};\bm{x}^{k}) -  \phi(\hat{\bx}^k;\bm{x}^{k})
			&=  \phi(\bm{x}^{k};\bm{x}^{k}) - \phi(\bx^*; \bx^*) + \phi(\bx^*; \bx^*) - \phi(\hat{\bx};\bm{x}^{*})\\
			&\quad + \phi(\hat{\bx};\bm{x}^{*}) - \phi(\hat{\bx}^k;\bm{x}^*) + \phi(\hat{\bx}^k;\bm{x}^*) -  \phi(\hat{\bx}^k;\bm{x}^{k})\\
			&\geq - |\phi(\bx^k; \bx^k) - \phi(\bx^*; \bx^*)| + |\phi(\bx^*; \bx^*) - \phi(\hat{\bx};\bm{x}^{*})|\\
			&\quad -|\phi(\hat{\bx}^k;\bm{x}^*) - \phi(\hat{\bx};\bm{x}^{*})| - \phi(\hat{\bx}^k;\bm{x}^*) - \phi(\hat{\bx}^k;\bm{x}^{k})\\
			&> -\delta / 4 + \delta -\delta / 4 - \delta / 4= \delta / 4,
		\end{aligned}
	\end{equation}
	which contradicts~\eqref{delta is small}. By~\eqref{eq: Contra_Delta_p}, $\hat{\bx}^k$ is feasible for the $k$th subproblem~\eqref{sub.1} and has lower objective than $\bm{x}^{k+1}$, and this contradicts the fact that $\bm{x}^{k+1}$ is optimal to~\eqref{sub.1}. Therefore, by~\Cref{lem.well.2}(a), $\bm{x}^{*}$ is first-order optimal for \eqref{lp.prob}.

	Case (ii): we consider two subcases that $\|\bx^*\|_p^p < r$ and $\|\bx^*\|_p^p = r$. 
	
	(ii)-(a): Suppose $\|\bx^*\|_p^p < r$ with   $\epsilon^*= c\Big(\tfrac{ r -  \|\bx^*\|_p^p  }{|\Acal(\bx^*)| + 1}\Big)^{\frac{1}{p}} > 0$. By \Cref{lem.well.2}(a), we can assume by contradiction that  the subproblem \eqref{sub.2}  at $\bx^*$ has 
	a stationary point $\tilde{\bm{x}} \ne \bx^*$ satisfying 
	\begin{equation}\label{eq: Delta_p_2}
		\Delta  \phi(\tilde{\bm{x}}; \bx^*) = \phi(\bx^*; \bx^*) - \phi(\tilde{\bm{x}}; \bx^*) \ge \delta > 0.
	\end{equation}
	Now consider subsequence $\hat{\Scal}_2 \subset \Scal_2$ such that $\{\bm{x}^{k}\}_{k \in \hat{\Scal}_2}\to \bm{x}^{*}$  and $\{\epsilon^{k}\}_{k \in \hat{\Scal}_2} \to \epsilon^*$. By \Cref{lem.delta.0}, for sufficiently large $k\in\hat{\Scal}_2$, we have
	\begin{equation}\label{eq: delta_p2_xk}
		\phi(\bm{x}^{k+1};\bm{x}^{k}) > \phi(\bm{x}^{k};\bm{x}^{k}) - \delta/4.
	\end{equation}
	Notice that $\phi(\bx^k; \bx^k) - \phi(\bx^*; \bx^*)  = f(\bx^k) - f(\bx^*)$. By the continuity of $f$, there exists sufficiently large $k \in \hat{\Scal}_2$ such that 
	\begin{equation}\label{eq: Delta_p2_xk}
		|\phi(\bx^k; \bx^k) - \phi(\bx^*; \bx^*)|  < \delta/4.
	\end{equation}
	Now consider projecting $\tilde{\bm{x}}$ onto the feasible region of~\eqref{sub.2} at $\bx^k$, and   denote the projection point as $\tilde{\bm{x}}^{k}$. By the continuity of $p$, there exists sufficiently large $k \in \hat{\Scal}_2$ such that 
	\begin{equation}\label{eq: Delta_tilde_xk}
		\begin{aligned}
			|\phi(\tilde{\bx}^k;\bm{x}^*) - \phi(\tilde{\bx}^k;\bm{x}^{k})| < \delta / 4\quad\text{and}\quad |\phi(\tilde{\bx}^k;\bm{x}^*) - \phi(\tilde{\bx};\bm{x}^{*})| < \delta/4.
		\end{aligned}
	\end{equation}
	Combining~\eqref{eq: Delta_p_2},~\eqref{eq: delta_p2_xk},~\eqref{eq: Delta_p2_xk} and~\eqref{eq: Delta_tilde_xk}, for sufficiently large $k\in\hat{\Scal}_2$, we have
	\begin{equation}\label{eq: ii_conclusion1}
		\begin{aligned}
			\phi(\bm{x}^{k};\bm{x}^{k}) -  \phi(\tilde{\bx}^k;\bm{x}^{k})
			&=  \phi(\bm{x}^{k};\bm{x}^{k}) - \phi(\bx^*; \bx^*) + \phi(\bx^*; \bx^*) - \phi(\tilde{\bx};\bm{x}^{*})\\
			&\quad + \phi(\tilde{\bx};\bm{x}^{*}) - \phi(\tilde{\bx}^k;\bm{x}^*) + \phi(\tilde{\bx}^k;\bm{x}^*) -  \phi(\tilde{\bx}^k;\bm{x}^{k})\\
			&\geq - |\phi(\bx^k; \bx^k) - \phi(\bx^*; \bx^*)| + |\phi(\bx^*; \bx^*) - \phi(\tilde{\bx};\bm{x}^{*})|\\
			&\quad -|\phi(\tilde{\bx}^k;\bm{x}^*) - \phi(\tilde{\bx};\bm{x}^{*})| - \phi(\tilde{\bx}^k;\bm{x}^*) - \phi(\tilde{\bx}^k;\bm{x}^{k})\\
			&> -\delta / 4 + \delta -\delta / 4 - \delta / 4= \delta / 4,
		\end{aligned}
	\end{equation}
	contradicting~\eqref{eq: delta_p2_xk}. This indicates that $\tilde{\bm{x}}^{k}$ is  feasible for the $k$th subproblem~\eqref{sub.2} and has lower objective than $\bm{x}^{k+1}$. Obviously, this contradicts the optimality of $\bm{x}^{k}$ for~\eqref{sub.2} at the $k$th iteration. Therefore, by \Cref{lem.well.2}(b), $\bm{x}^{*}$ is first-order optimal for \eqref{lp.prob}.

	(ii)-(b): Suppose $\|\bx^*\|_p^p = r$ with   $\epsilon^*= c\Big(\tfrac{ r -  \|\bx^*\|_p^p  }{|\Acal(\bx^*)| + 1}\Big)^{\frac{1}{p}}= 0$. We know $\Ical(\bx^*) \ne \emptyset$.    Consider subsequence $\hat{\Scal}_2$ such that $\{\bm{x}^{k}\}_{k \in \hat{\Scal}_2} \to \bm{x}^{*}$.  
	By \eqref{lem. iterates_dimishes},  for any $i\in\Ical(\bx^*)$,  there exists $\bar{k} \in \hat{\Scal}_2$ such that $\{x_{i}^{k} \}_{k \geq \bar{k}, k\in\hat{\Scal}_2}$ and $\{x_i^{k+1}\}_{k \geq \bar{k}, k\in\hat{\Scal}_2}  $ are bounded away from 0, meaning $\Ical(\bm{x}^{*}) \subset \Ical(\bm{x}^{k})$ for  sufficiently 
	large  $k\in\hat{\Scal}_2$.   By~\eqref{kkt.sub2}, we have
	\begin{equation*}
		\begin{aligned}
			\lambda^{k+1}&=-\frac{  x_i \nabla_i f(\bm{x}^k)+\beta  x_i (x_i^{k+1}-x_i^k)  }{ p | x_i^k|^{p-1}|x_i^{k+1}|  },  \  i\in\Ical(\bx^*). 
		\end{aligned}
	\end{equation*} 
	Obviously,  $\{ \lambda^{k+1} \}_{k \geq \bar{k}, k\in\hat{\Scal}_2}$ are bounded above.  Let $\lambda^*$ be a limit point of $\{\lambda^{k+1}\}$ with
	subsequence   $\{\lambda^{k+1}\}_{\tilde{\Scal}_2} \to \lambda^*$.
	By \Cref{lem. iterates_dimishes},  $\{\bx^{k+1}\}_{\tilde{\Scal_2}}\to\bx^*$.   
	It follows that for any $i\in \Ical(\bx^*)$,  
	\begin{equation}\label{eq: satisfy_optimality} 
		\begin{aligned} 
			0 = 	\nabla_i f(\bx^{*}) x_{i}^{*}+\lambda^{*} p |x_i^{*}|^{p}
			=  \lim_{k\in \tilde{\Scal}_2 \atop k\to\infty}
			\nabla_i f(\bx^k) x_i^{k+1} +\lambda^{k+1} p |x_i^{k+1} |^{p}.
		\end{aligned} 
	\end{equation}
	This proves that $\bm{x}^{*}$ is first-order optimal for \eqref{lp.prob}.
\end{proof}

\section{Numerical Experiments}\label{Sec: Numerical Study}
In this section, we test the proposed algorithm IR1B on synthetic data and real-world data to demonstrate the practicability and effectiveness for solving~\eqref{lp.prob}. With this target, we test the sparse signal recovery problem that stems from compressive sensing and also test the logistic regression problem in statistical machine learning. Of all the tests, we choose the $\ell_0$ ball or the $\ell_1$ ball constraints to apply to the same testing problems as the benchmark to show the quality of the solutions of these problems,  since the existing approach for solving~\eqref{lp.prob} is limited. We implement all codes using Python and run all experiments on a laptop under Ubuntu with 7.5 GB main memory and Intel Core i7-7500U processor (2.70GHz $\times$ 4). 
\subsection{$\ell_{p}$-constrained Least squares on synthetic data}
Consider a underdetermined sensing matrix $\bm{A} \in \mathbb{R}^{m\times n}$ and a noisy measurement vector $\bm{y}\in\mathbb{R}^{m}$. The unknown signal $\bm{x}^{\dagger}$ with $d$-nonzeros to be estimated satisfies $\bm{y} = \bm{A}\bm{x}^{\dagger}$. Then the $\ell_{p}$-constrained signal recovery problems can be formulated as
\begin{equation}\label{eq:recovery_problem}
	\begin{aligned}
		\min_{\bm{x}\in\mathbb{ R}^n} \quad&\frac{1}{2}\|\bm{A}\bm{x}-\bm{y}\|_2^2\\
		\text{s.t.} \quad&\|\bm{x}\|_p^p\le r,
	\end{aligned}
\end{equation}
where $p\in(0,1]$. In particular, the plain Iterative Hard Thresholding (IHT)~\cite{blumensath2009iterative} algorithm and the method described in~\cite{zhang2020inexact} (abbreviated as GPM) are well-known for solving~\eqref{eq:recovery_problem} when $p = 0$ and $p = 1$, respectively. 

We first state the way to generate simulation data. Concretely,  $\bm{x}_{\text{ori}}$ represents the original  $d$-sparse vector. In this context, $m$ is the number of noisy measurements, and we set $m \in [50,1000]$ and increase it by $50$ for each comparison. Given each $m$, we randomly generate $d$ nonzero elements of $\bm{x}_{\text{ori}}$ by assigning their value with $+1$ and $-1$ with equal probability, and each entry of $\bm{A}$ is sampled from a standard Gaussian distribution. Then we form $\bm{b}=\bm{A}\bm{x}_{\text{ori}}+\bm{\eta}$ with $\bm{\eta}$ being Gaussian noise with $0$ mean and $10^{-2}$ standard deviation.

In this experiment, we set $c = 0.95$, $r=s$ and initialize $\bm{x}^{0} = 0.9(\frac{d}{\Vert\bm{\nu}\Vert_1}\bm{\nu})^{1/p}$ such that $\Vert\bm{x}\Vert_{p}^{p}\leq r$, where each entry of $\bm{\nu}\in\mathbb{R}^{n}$ are uniformly sampled over the interval $[0,1)$. The stepsize $1/\beta$ for IR1B and $\beta$ for IHT and GPM is used, where $\beta = 1.1L$, $L$ represents the Lipschitz constant of the objective in~\eqref{eq:recovery_problem} and $\lambda_{\text{max}}(\bm{A}^{T}\bm{A})$ denotes the largest eigenvalue of $\bm{A}^{T}\bm{A}$. For IR1B,  we determine that a point is on the boundary of the $\ell_p$ ball if $|r-\|\bm{x}\|_p^p |\le 10^{-8}$. IR1B, IHT and GPM are terminated if $\|\bm{x}^{k+1}-\bm{x}^k\|_2\le \text{tol}$ with $\text{tol} = 10^{-5}$. In addtion, we declare a success for the test if $\|\bm{x}^*-\bm{x}^{\dagger}\|_2/\|\bm{x}^{\dagger}\|_2 < 10^{-3}$ is satisfied, where $\bm{x}^{*}$ denotes the optimal solution output by three algorithms. With the different values of $p$, we compare three algorithms with respect to the empirical probability of success defined by the ratio of the number of successes to the total $50$ runs for each comparison. The curve is shown in~\Cref{fig:empirical_success}, and each presented result is the average of $50$ independent runs.
\begin{figure}[htbp]
	\centering
	\includegraphics[width=3.6in]{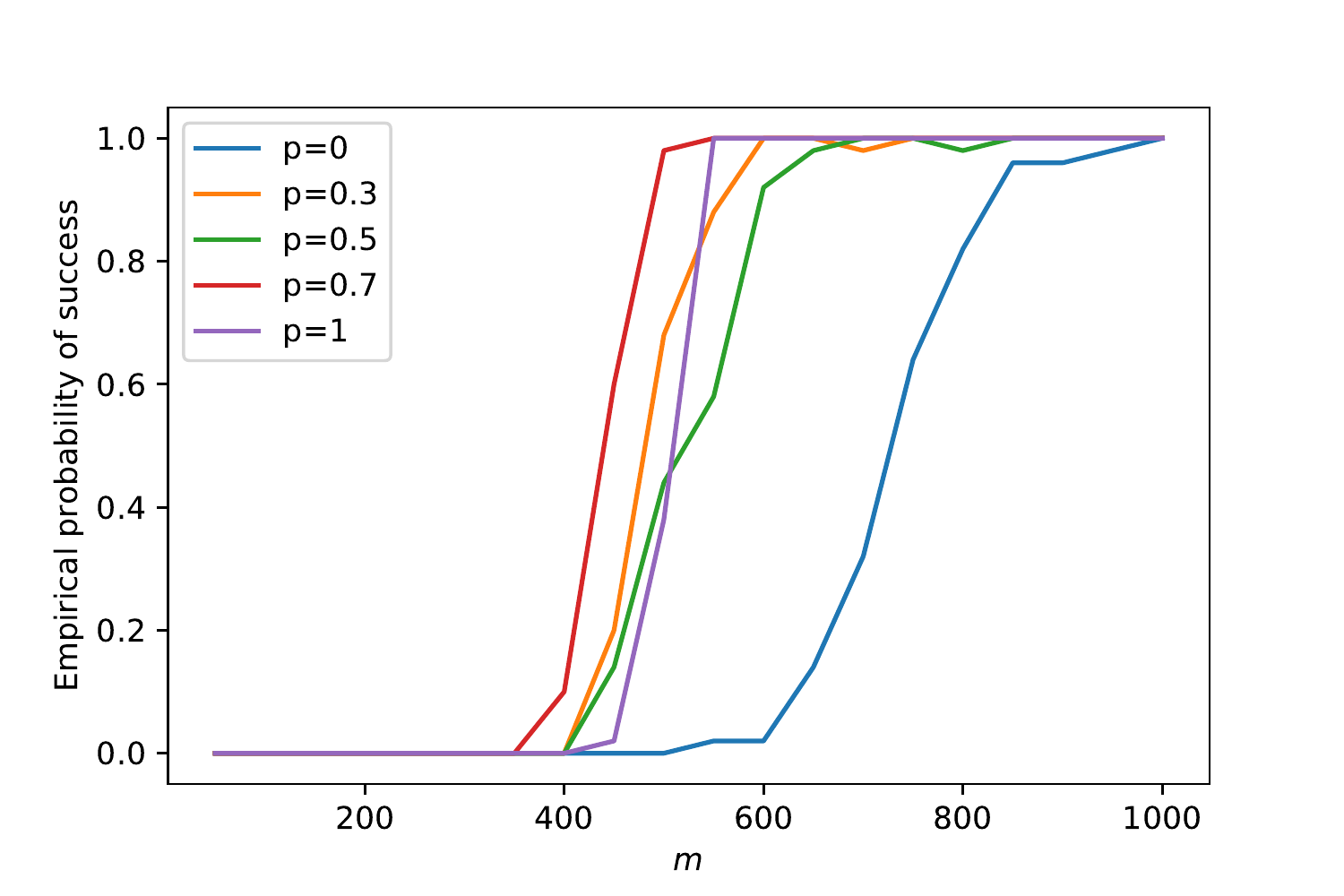}
	\caption{The empirical probability of success versus $m$ for various $\ell_{p}$ ball constraints.}
	\label{fig:empirical_success}
\end{figure}


As observed in~\Cref{fig:empirical_success}, with $p \in \{0.3,0.5,0.7\}$, IR1B successfully solves~\eqref{eq:recovery_problem}. In particular, we can see that $\ell_{0.7}$ ball has $5$ successful recoveries and requires the minimal number of observations, i.e., $m = 400$. On the other hand, we observe that $\ell_{0.3}$ ball and $\ell_{0.5}$ ball yield more successful recoveries than $\ell_0$ ball and $\ell_1$ ball when $m = 450$. When $m > 600$, there are a few successful recoveries for $\ell_0$ ball while other $\ell_{p}$ balls achieve success for each run. Overall, the $\ell_p$ ball with $p \in (0,1)$ is superior to the $\ell_0$ ball and $\ell_1$ ball for recovering the original sparse signal if the number of observation elements is limited. 


\subsection{$\ell_{p}$-constrained logistic regression on real-world data}
The focus of this test is to consider the logistic regression model which is a popular classification approach in the context of supervised learning. Given a collection of data pairs $\{(\bm{x}^{(i)},y^{(i)})\}_{i=1}^{m}$, where $\bm{x}^{(i)} \in \mathbb{R}^{n}$ represents a feature vector and $y^{(i)}\in \{-1,+1\}$ denotes a binary instance label. Then, the $\ell_{p}$ ball constrained logistic regression problem is 
\begin{equation}\label{eq: Logistic_model}
	\begin{aligned}
		\min_{\bm{\theta}}\quad &\sum_{i=1}^{n}  \log\left(  {1 + \exp(-\bm{\theta}^{T}\bm{x}^{(i)})}\right) \\
		\text{s.t.}\quad &\Vert\bm{\theta}\Vert_{p}^{p} \leq r.
	\end{aligned}
\end{equation}

We evaluate the performance of IR1B on the Wisconsin breast cancer dataset~\cite{blake1998uci}, which contains $569$ instances and $30$ features. For compactness, we denote the data matrix as $\bm{X} = (\bm{x}^{(1)}; \ldots; \bm{x}^{(m)})^{T}\in\mathbb{R}^{569\times 30}$. In this test, we randomly split $\bm{X}$ into the train subset $\bm{X}_{\text{train}}$ and test subset $\bm{X}_{\text{test}}$ where  the test subset size accounts for $40\%$ of the total data set size. 

Let $\bm{\theta}^{\dagger}$ be the optimal solution returned by IR1B and $\hat{\bm{y}}$ be the predicted class label, i.e., $\hat{y}^{(i)} = 1$ if $(1 +\exp(-\bm{\theta}^{T}\bm{x}^{(i)}))^{-1} \geq 0.5$; otherwise $\hat{y}^{(i)} = 0$. We initialize $\bm{\theta}^{0} = \bm{0}$ and set $\beta = 1.1L$, where $L = 0.25\lambda_{\text{max}}(\bm{X}^{T}\bm{X})$ represents the Lipschitz constant of the objective in~\eqref{eq: Logistic_model}, and all other parameters remain the same as the first experiment.


We plot the curve of two errors versus a range of values of $r$ varying from $2$ to $35$, as shown in~\Cref{fig:test_acc}.

\begin{figure}[htbp]
	\centering
	\includegraphics[width=3.6in]{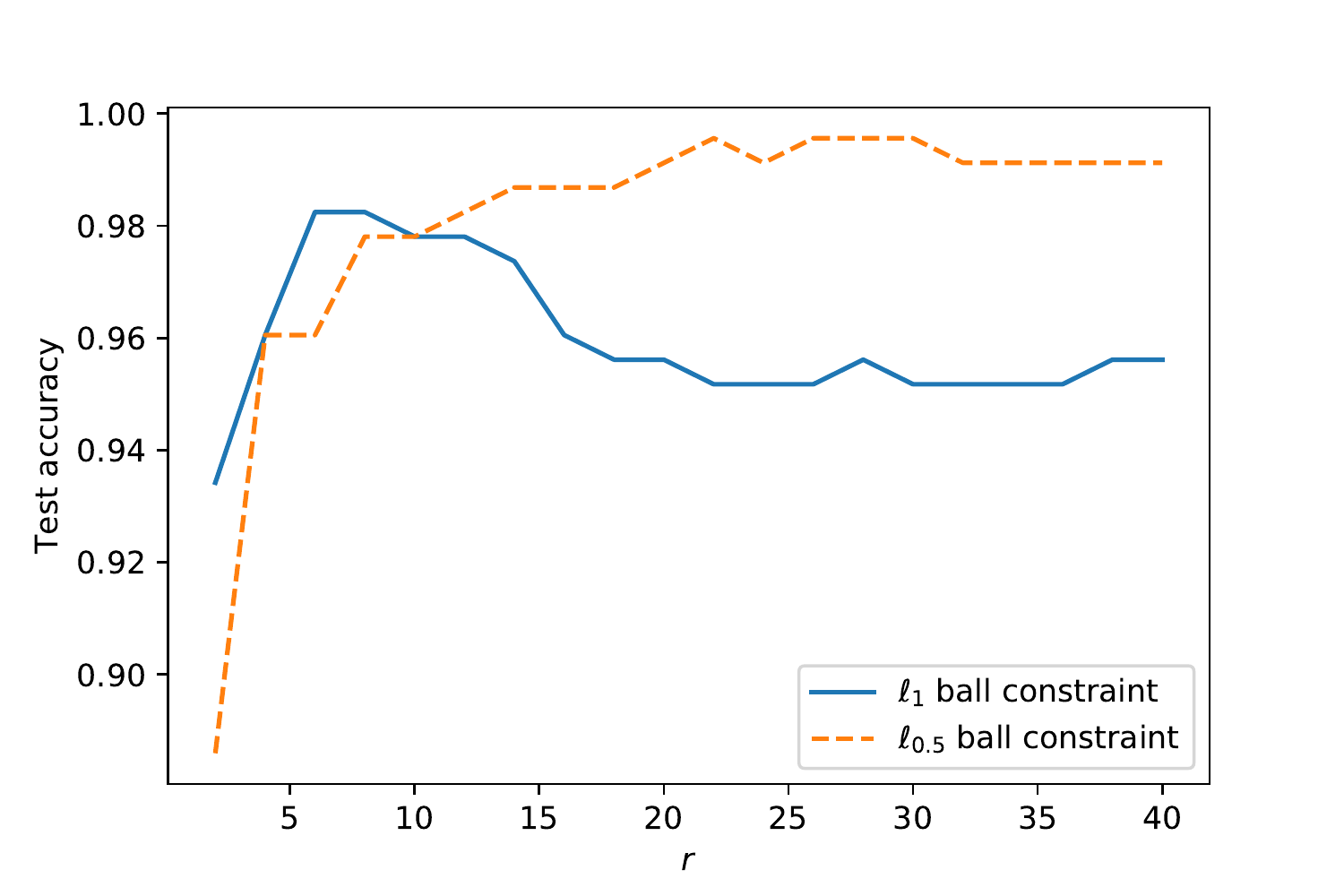}
	\caption{The empirical probability of success versus $m$ for   $\ell_{0.5}$ ball and $\ell_1$ ball  constraints.}
	\label{fig:test_acc}
\end{figure}

\Cref{fig:test_acc} illustrates the prediction accuracy on the test data for $\ell_{0.5}$ ball constraint and $\ell_1$ ball constraint. As observed, increasing $r$ gradually improves the prediction accuracy; we can see that $\ell_{p}$ ball constraint yields a larger prediction accuracy over the test data than that of $\ell_1$ ball constraint. These observations suggest that the $\ell_{p}$ ball constraint can achieve   better generalization error than the $\ell_1$ ball constraint.

\vfill

\section{Conclusions}\label{Sec: Conclusion}
In this paper, we proposed and analyzed   the $\ell_p$ ball constrained nonlinear optimization problems. The proposed iteratively reweighted $\ell$ ball method is simple to implement and only needs to solve a weighted $\ell_1$ ball projection problem. We proved the proposed the global convergence of our proposed algorithm  to the first-order stationary point of the original problem from any feasible initial point. The effectiveness of the proposed algorithm was demonstrated on the sparse signal recovery problems and logistic regression problems.

\appendix
\section{Appendix:  Proof of~\cref{thm:FirstOrderNecessaryCondition}: Characterizing the elements of the Fr\'echet normal cone}\label{Appendix_1}
\begin{proof}
	This proof is an direct extension of~\cite[Proposition 2.4]{yang2021towards} for a general objective function $f$. For any $\bm{x}\in \Omega$ and $\bm{x}$ sufficiently close to $\bar{\bm{x}}$, we have
	\begin{equation*}
		\begin{aligned}
			0 &\ge \sum_{i=1}^{n}|x_i|^p-\sum_{i=1}^{n}|\bar{x}_i|^p = \sum_{i\in \mathcal{A}(\bar{\bm{x}})}(|x_i|^p-0)+\sum_{i\in \mathcal{ I }(\bar{\bm{x}})}(|x_i|^p-|\bar{x}_i|^p)\\
			&=\sum_{i\in \mathcal{A}(\bar{\bm{x}})}|{x}_i|^{p-1}\text{sgn}(x_i)x_i+\sum_{i\in \mathcal{ I }(\bar{\bm{x}})}p|\bar{x}_i|^{p-1} (| x_i | - | \bar{x}_i|)+o(\| |\bm{x}_{\Ical(\bar{\bm{x}})}| - |\bar{\bm{x}}_{\Ical(\bar{\bm{x}})}|\|_2)\\
			&\ge  \sum_{i\in \mathcal{A}(\bar{\bm{x}})}(|{x}_i|^{p-1}\text{sgn}(x_i)-\eta_{i})x_i+\sum_{i\in \mathcal{A}(\bar{\bm{x}})}\eta_{i}x_i\\ &\quad+\sum_{i\in \mathcal{ I }(\bar{\bm{x}})}p|\bar{x}_i|^{p-1}\text{sgn}(\bar{x}_i)(x_i-\bar{x}_i)+o(\| |\bm{x}_{\Ical(\bar{\bm{x}})}| - |\bar{\bm{x}}_{\Ical(\bar{\bm{x}})}|\|_2),
		\end{aligned}
	\end{equation*}
	where the second equality is obtained by the Taylor series approximation of $|x_i|^p$ at $\bar{x}_i$ and the second inequality is 
	by the convexity of $|\cdot |$. 
	It then follows that 
	\begin{equation*}
		\begin{aligned}
			\frac{\langle\bm{\eta},\bm{x}-\bm{\bar{x}}\rangle}{\|\bm{x}-\bar{\bm{x}}\|_2}
			&=\frac{\sum\limits_{i\in \mathcal{A}(\bar{\bm{x}})}\eta_{i}x_i+\sum\limits_{i\in \mathcal{I}(\bar{\bm{x}})}p|\bar{x}_i|^{p-1}\text{sgn}(\bar{x}_i)(x_i-\bar{x}_i)}{\|\bm{x}-\bar{\bm{x}}\|_2}\\
			&\le \frac{\sum\limits_{i\in \mathcal{A}(\bar{\bm{x}})}(\eta_{i}-|{x}_i|^{p-1}\text{sgn}(x_i))x_i\!+\!o(\||\bm{x}_{\mathcal{I}(\bar{\bm{x}})}|-|\bar{\bm{x}}_{\Ical(\bar{\bm{x}})}|\|_2)}{\|\bm{x}-\bar{\bm{x}}\|_2}\\
			&\le\frac{\sum\limits_{i\in \mathcal{A}(\bar{\bm{x}})}(\eta_{i}\text{sgn}(x_i)-|{x}_i|^{p-1})|x_i|}{\|\bm{x}-\bar{\bm{x}}\|_2}+\frac{o(\|\bm{x}_{\Ical(\bar{\bm{x}})}-\bar{\bm{x}}_{\Ical(\bar{\bm{x}})}\|_2)}{\|\bm{x}_{\Ical(\bar{\bm{x}})}-\bar{\bm{x}}_{\Ical(\bar{\bm{x}})}\|_2}.
		\end{aligned}
	\end{equation*}
	As $x\to \bar{x}$, we have $\frac{o(\|\bm{x}_{\Ical(\bar{\bm{x}})}-\bar{\bm{x}}_{\Ical(\bar{\bm{x}})}\|_2)}{\|\bm{x}_{\Ical(\bar{\bm{x}})}-\bar{\bm{x}}_{\Ical(\bar{\bm{x}})}\|_2}\rightarrow0$. 
	Furthermore, we have $\eta_{i}\text{sgn}(x_i)-|{x}_i|^{p-1}<0$ for all $i\in \mathcal{A}(\bar{\bm{x}})$. 
	It follows that 
	$\limsup\limits_{\bx\rightarrow \bar{x}, x\in \Omega}\frac{\langle\bm{\eta},\bm{x}-\bm{\bar{x}}\rangle}{\|\bm{x}-\bar{\bm{x}}\|_2}\le 0$, completing the proof.
\end{proof}

\bibliographystyle{ieeetr}
\bibliography{references}

\end{document}